\documentclass[11pt,a4paper]{article}
\usepackage{latexsym,  bm}

\usepackage{amsfonts}
\usepackage{amsthm}
\usepackage{amsmath}
\usepackage{amssymb}
\usepackage{graphicx}
\usepackage{layout}
\usepackage{epsfig}
\usepackage{indentfirst}
\newcommand{\va}{{\mathbb \varepsilon}}
\newcommand{\R}{{\mathbb R}}

\newcommand{\z}{{\mathbb Z}}

\newcommand{\La}{\Lambda}

\newcommand{\dis}{\displaystyle}
\theoremstyle{plain}
\newtheorem{theorem}{Theorem}[section]

\newtheorem{lemma}{Lemma}[section]
\newtheorem{proposition}{Proposition}[section]

\numberwithin{equation}{section}
\newcommand{\beq}{\begin{equation}}
\newcommand{\eeq}{\end{equation}}

\begin{document}

\def\ST{\songti\rm\relax}
\def\HT{\bf\relax}
\def\FS{\fangsong\rm\relax}
\def\KS{\kaishu\rm\relax}
\def\REF#1{\par\hangindent\parindent\indent\llap{#1\enspace}\ignorespaces}
\def\pd#1#2{\frac{\partial#1}{\partial#2}}
\def\ppd#1#2{\frac{\partial^2#1}{\partial#2^2}}

\title  {{ The global existence and time-decay for the solutions of the fractional pseudo-parabolic equation}
\thanks {Research was supported by the Natural Science Foundation of China
(11101160,11271141) and  China Scholarship Council (201508440330).}}
\author
{\small { Lingyu Jin, Lang Li and Shaomei Fang}\\
\footnotesize  { Department of Mathematics,   South China Agricultural University,   Guangzhou 510642,China}}

\maketitle

 \noindent
 {\bf Abstract:}
We consider the Cauchy problem of fractional pseudo-parabolic equation on the whole space $\R^n,n\geq 1$. Here, the fractional order $\alpha$ is related to the diffusion-type source term behaving as the usual diffusion term  on the high frequency part. It has  a feature of regularity-gain and regularity-loss for $\alpha> 1$ and $0<\alpha < 1$, respectively. We establish the global existence and time-decay rates for small-amplitude solutions to the  Cauchy problem for $\alpha>0$. In the case that  $0<\alpha < 1$ , we introduce the time-weighted energy method to overcome the weakly dissipative property of the equation. 


\section{Introduction}
In this paper we consider the Cauchy problem of the following fractional pseudo-parabolic equation
\begin{eqnarray}\label{1.1}
\begin{cases}
u_t-m\Delta u_t+(-\Delta )^{\alpha}u=u^{\theta+1},\,\,&x\in\R^n ,t>0,\\
u(0,x)=u_0(x),\,\,&x\in\R^n.
\end{cases}
\end{eqnarray}

For $m>0, \alpha=1$, (\ref{1.1}) is called the pseudo-parabolic equation
 (refer to \cite{T},\cite{ST}) since the solutions of the initial-boundary value problem for a parabolic equation can be obtained as the the limit of some sequence of solutions of the Cauchy problems for the corresponding pseudo-parabolic equations. In addition, the initial-boundary value problem for the case $m=0$ is also well-posed for the pseudo-parabolic equation (1.1).
Furthermore, equation (1.1) can be regarded as a Sobolev-type  or a Sobolev-Galpern type equation. A large amount of physical phenomena  such as seepage of homogeneous fluids through a fissured rock(\cite{BZ}), aggregation of populations (\cite{L}) can be described by pseudo-parabolic equations.
 Ting, Showalter and Gopala Rao (refer to \cite{T}, \cite{ST}, \cite{GT}) investigated the initial-boundary value problem and established the existence and uniqueness of solutions.  From then on, considerable attention
has been paid to the study of nonlinear pseudo-parabolic equations,  including singular pseudo-parabolic equations and degenerate pseudo-parabolic equations (see \cite{L}, \cite{B}, \cite{CY},   \cite{FY}, \cite{KN}, \cite{KQ}, \cite{P}, \cite{SF}, \cite{K}). Since the fractional dissipation operator $(-\Delta)^{\alpha}$ is nonlocal and can be regarded as the infinitesimal generators of Levy stable diffusion processes, many scientists have found that it describes some physical phenomena more exact than integral differential equations( refer to \cite{PG}, \cite{GHX}, \cite{T2}, \cite{P1},  \cite{KS1}). More and more work has been devoted to the investigation of fractional differential equations( \cite{PG}, \cite{GHX},  \cite{CS}, \cite{Z}). Motivated by these results, we will mainly study the fractional pseudo-parabolic equations (\ref{1.1}).  Obviously, the equation (1.1) is of the regularity-gain type for $\alpha > 1$, and whereas of the regularity-loss type for $0<\alpha<1$. For equations with regularity-loss structure, a large amount of results has been established (refer to  \cite{D}, \cite{DR}, \cite{IH}).

   In this paper, we consider the fractional pseudo-parabolic equations for $\alpha>0$. The main purpose is to  obtain the well-posedness and large-time behavior of the Cauchy problem (1.1) for small-amplitude solutions.   We establish the global existence of solutions and the time-decay rate of solutions and their derivatives up to some order, where the extra regularity on initial data is required in the case of the regularity-loss type.

We now introduce some notations. 

 In what follows, we denote generic positive constants by $c$ and $C$ which may change from line to line. The Fourier transform
 $\hat{f}$  of a tempered distribution $f(x)$ on $\R^n$ is defined as
 \[\mathcal{F}[f(x)](\xi)= {\hat f(\xi)}=\frac{1}{(2\pi)^n
 	}\int_{R^n}f(x)e^{-i\xi\cdot x}dx.\]

We will denote the square root of the Laplacian $(-\Delta)^{\frac{1}{2}}$ by $\Lambda$ and obviously
\[\widehat{\Lambda f}=|\xi|\hat{f}(\xi).\]
Denote $H^{s}(\R^n)$ the general Sobolev space with the norm $$\|f\|^2_{H^s}=\int_{\R^n}(1+|\xi|^2)^{s}|\hat f|^2d\xi.$$
For $s=0$, $H^0(\R^n)=L^2(\R^n)$.
 Obviously $H^{s}(\R^n)$ is the Hilbert space, and with the inner product
$$(f,g)_{H^{s}}=\int_{\R^n}\hat f (\xi)\overline{\hat{g}(\xi)}(1+|\xi|^2)^sd\xi.$$

In this paper, we consider the Cauchy problem (\ref{1.1}) under  some conditions on $\theta$, $\alpha,n$, we are mostly interested in the large time behavior of the solutions.

 Now we introduce the main theorems in this paper.
  \begin{theorem}
  	Let $s> \frac{n}{2} ,\,\,\alpha\geq 1,\theta>\frac{4\alpha}{n},\theta\in \z
  $.  Assume that $u_0\in H^s(\R^n)\bigcap L^1(\R^n)$ and let $E_0=\|u_0\|_{H^s}+\|u_0\|_{L^1}$. Then there exists a small positive constant $\delta_0$ such that for  $E_0\leq \delta_0$, (1.1)  has a global solution  satisfying
  	the following decay estimate
  	$$\|\Lambda ^l  u(t)\|_{L^2} \leq c E_0(1+t)^{-\frac{n}{4\alpha}-\frac{l}{2\alpha}},\,\,\forall\,\,\, 0\leq l\leq s.$$
   \end{theorem}
  
 \begin{theorem}
Assume that $0<\alpha<1$, $ \theta>\frac{4\alpha}{n},\theta\in \z$ and
 $[\frac{s}{2\bar{\alpha}}]\geq\frac{n}{2\alpha}+\frac{3}{2}\bar{\alpha}.$  Let $E_0=\|u_0\|_{H^s}+\|u_0\|_{L^1}$ for $u_0\in H^s(\R^n)\bigcap L^1(\R^n)$. Then there exists a small positive constant $\delta_0$ such that for $E_0\leq \delta_0$, (1.1)  has a global solution  satisfying
 $$\|\Lambda ^l  u(t)\|_{L^2} \leq c E_0(1+t)^{-\frac{n}{4\alpha}-\frac{l}{2\alpha}}, \text{ for } 0\leq l\leq N_0$$
 with $\bar \alpha=1-\alpha$, $ N_0=\alpha \min\{s-\frac{n}{2\alpha}\bar{\alpha}, ([\frac{s}{2\bar{\alpha}}]-1)\bar{\alpha}-\frac{n}{2\alpha}\bar{\alpha}+2\}$.
 \end{theorem}

The time-decay rate in the case of $\alpha<1$ is obtained for $u$ and its lower-order derivatives only, but not for all derivatives up to $s$, whereas in the case of $\alpha> 1$ is obtained for all derivatives up to $s$. Thus  (1.1)  is of the regularity-gain type for $\alpha> 1$ and the regularity-loss type for $0<\alpha<1$. 
The difference between Theorem 1.1 and Theorem 1.2 arises from the property of the high frequency part of the linearized equation (3.1) (refer to Lemma 3.2). For the regularity-loss type ($0<\alpha<1$), we mainly take the time-weighted energy method to overcome the weakly dissipative property of the equation (refer to \cite{D}, \cite{DR}, \cite{IH}). In addition, multiplying (\ref{1.1}) by $u$ and integrating, 
\begin{eqnarray}\label{1.2j}
\frac{d}{dt}\bigl (\|u\|^2_{L^2}+m\|\La u\|^2_{L^2})+\|\La^{\alpha/2}u\|^2_{L^2}=\int_{\R^n}u^{\theta+2}dx,
\end{eqnarray}
since the special structure of (\ref{1.2j}), we can not have an estimate of the lower order term $\|u\|_{L^2}$ as usual. Thus the right side of (1.2) can not be controlled by standard energy method. We employ the approach of a long wave and short wave method (refer to \cite{D4}) to overcome this difficulty.
The rest of paper is organized as follows. In Section 2, we give some preliminary lemmas. In Section 3, we establish some decay estimates for the linearized system. Finally we get decay estimates for the solution of (\ref{1.1}) for $\alpha\geq 1$ and $0 < \alpha < 1$ respectively, thus the global existence of the solutions in both cases is proved in Section 4.

\section{Preliminaries}
In this section, we give some preliminary lemmas.

\begin{lemma}\label{l2.1}
	(Refer to \cite{S2}, \cite{J}) Assume that $1\leq p,q,r\leq +\infty, \frac{1}{r}=\frac{1}{p}+\frac 1 q$ and $g\,\,,h\in W^{l,q}(\R^n)\bigcap L^{p}(\R^n)$, then
\begin{eqnarray}
\|\La ^l(gh)\|_{L^r}\leq C(\|g\|_{L^p}\|\La^l h\|_{L^q}+\|\La^l g\|_{L^q}\|h\|_{L^p}).
\end{eqnarray}
\end{lemma}

\begin{lemma}\label{l2.4}
(Galiardo-Nirenberg inequality, \cite{H, T1})
Suppose that $u\in L^q(\mathbb R^n)\cap W^{m,r}(\mathbb R^n)$, where $1\leq q, r\leq\infty$. Then there exists a constant $C>0$, such that
\begin{eqnarray}\|D^j	u\|_{L^p}\leq C\|D^mu\|^a_{L^r}\|u\|^{1-a}_{L^q},
\end{eqnarray}
where
$$\frac{1}{p}=\frac{j}{n}+a\left(\frac{1}{r}-\frac{m}{n}\right)+(1-a)\frac{1}{q},$$
$1\leq p\leq\infty$, $j$ is an integer, $0\leq j\leq m, j/m\leq a\leq 1$. If $m-j-n/r$ is a nonnegative integer, then the inequality holds for $j/m\leq a<1$.
\end{lemma}

\begin{lemma}\label{l2.9}
Assume that $u\in H^l(\R^n)\bigcap L^\infty$, then
	\begin{eqnarray}\label{2.3j}
	\|\La^l (u^{\theta+1})\|_{L^2}\leq C\|u\|^\theta_{L^\infty}\|\Lambda^lu\|_{L^2},\forall\, \theta>0,\theta\in \mathbb{Z}.
	\end{eqnarray}

\end{lemma}
\begin{proof}
From Lemma \ref{l2.1}, we have (\ref{2.3j}) for $\theta=1$.
If  (\ref{2.3j}) holds  for $\theta=k$, then
 \begin{eqnarray}
\|\La ^l(u^ku)\|_{L^2}\leq C\|u\|_{L^\infty}^k\|\La^l u\|_{L^2}.
\end{eqnarray}
Then for $\theta=k+1$, \begin{eqnarray}
\|\La ^l(u^{k+1}u)\|_{L^2}\leq C(\|u\|_{L^\infty}^{k+1}\|\La^l u\|_{L^2}+\|\Lambda^l u^{k+1}\|_{L^2}\|u\|_{L^\infty})
\leq c\|u\|^{k+1}_{L^\infty}\|\Lambda^l u\|_{L^2}.
\end{eqnarray}
Using an induction argument, we have (\ref{2.3j})  for the general case.
\end{proof}
\section{Decay estimates for the linearized system}
In this section, we study the decay property of solutions to the linearized equation
\begin{eqnarray}\label{2.1}
\begin{cases}
u_t-m\Delta u_t+(-\Delta )^{\alpha}u=0,\,\,& x\in\R^n,\,t>0,\\
u(0,x)=u_0(x),\,\,&x\in\R^n.
\end{cases}
\end{eqnarray}
Applying the Fourier transform to (\ref{2.1}),  we arrive at the expression \[\hat{u}(t,\xi)=e^{-\frac{|\xi|^{2\alpha}}{1+m|\xi|^2}t}\hat{u}_0(\xi).\]
Define the Green function of the equation (\ref{2.1}) as \[G(x,t)=\mathcal{F}^{-1}e^{-\frac{|\xi|^{2\alpha}}{1+m|\xi|^2}t}.\]

Next we are going to obtain some properties of the Green function $G(x,t)$. In order to use the composition  method of long wave and short wave in Section 4. We need to obtain the estimates the long wave (high frequency) part and the short wave (low frequency) part of $G(x,t)$ respectively.
Let \begin{eqnarray}\label{j3.3}
\chi(\xi)=\begin{cases}
1,&|\xi|\leq R,\\ 
0,&|\xi|\geq 2R 
\end{cases}
\end{eqnarray}
be a smooth cut-off function for some fixed constant $0<R<1$.
Let \begin{eqnarray}\label{3.4jj}
\begin{split}
\hat G_L(\xi,t)=\chi(\xi)\hat G(\xi,t),\,\,\,G_L(x,t)=\chi(D)G(x,t),\\
\hat G_H(\xi,t)=(1-\chi(\xi))\hat G(\xi,t),G_H(x,t)=(1-\chi(D)) G(x,t)
\end{split}
\end{eqnarray}
 where $\chi(D)$ is the operator with the symbol $\chi(\xi)$, $G_H(x,t)$ is the long wave part of $G(x,t)$ and $G_L(x,t)$ is the short wave part of $G(x,t)$.


\begin{lemma}\label{l3.1}
If $\alpha>0,\,l>0$,  there exists a  constant $C>0$, such that 
\begin{eqnarray}
\|\Lambda ^l G_L(x,t)*\phi\|_{L^2}\leq C(1+t)^{-\frac{l}{2\alpha}}\|\phi\|_{L^2},\,\,\forall \phi\in L^2(\R^n),
\end{eqnarray}
 \begin{eqnarray}
\|\Lambda ^l G_L(x,t)*\phi\|_{L^2}\leq C(1+t)^{-\frac{n}{4\alpha}-\frac{l}{2\alpha}}\|\phi\|_{L^1},\,\,\forall \phi\in L^1(\R^n) .\end{eqnarray}
\end{lemma}
\begin{proof}
For $|\xi|\leq R$, we have
\begin{eqnarray}
\frac{|\xi|^{2\alpha}}{1+m|\xi|^2}\geq c|\xi|^{2\alpha},
\end{eqnarray}
then
\begin{eqnarray}\label{3.7109}
e^{-\frac{|\xi|^{2\alpha}}{1+m|\xi|^2}t}\leq e^{-c|\xi|^{2\alpha }t}.
\end{eqnarray} 
On the other hand, since  
\[ \sup_{|\xi|\leq R}|\xi|^{2l}e^{-c|\xi|^{2\alpha}t}\leq C (1+t)^{-l/\alpha},\,\text{ for } t\leq 1,\]
and 
\[\sup_{|\xi|\leq R}|\xi|^{2l}e^{-c|\xi|^{2\alpha}t}\leq C\sup_{\eta\in \R^n }|\eta|^{2l}e^{-c|\eta|^{2\alpha}}t^{-l/\alpha}\leq C (1+t)^{-l/\alpha},\,\text{ for } t> 1,\]
it follows that
\begin{eqnarray}\label{3.9jjjj}
\sup_{|\xi|\leq R}|\xi|^{2l}e^{-c|\xi|^{2\alpha}t}\leq C (1+t)^{-l/\alpha},\,\text{ for } t>0.
\end{eqnarray}
Using (\ref{3.7109}) and (\ref{3.9jjjj}), we can obtain that
\begin{eqnarray}
\begin{split}
\|\Lambda^l G_L(x,t)*\phi\|_{L^2}^2 &\leq \int_{|\xi|\leq R}|\xi|^{2l} e^{-c|\xi|^{2\alpha}t}|\hat \phi|^2d\xi\\&
\leq \|\hat \phi\|_{L^\infty}^2 \int_{|\xi|\leq R}|\xi|^{2l} e^{-c|\xi|^{2\alpha}t}d\xi
\\&\leq c(1+t)^{-\frac{n}{2\alpha}-\frac{l}{\alpha}}\|\phi\|_{L^1}^2.
\end{split}
\end{eqnarray}

Finally, from (\ref{3.9jjjj}), we can also have
\begin{eqnarray}
\begin{split}
\|\Lambda^l G_L(x,t)*\phi\|_{L^2}^2 &\leq \int_{|\xi|\leq R}|\xi|^{2l} e^{-c|\xi|^{2\alpha}t}|\hat \phi|^2d\xi\\&
\leq \|\hat \phi\|_{L^2}^2 \sup_{|\xi|\leq R}|\xi|^{2l}e^{-c|\xi|^{2\alpha}t}
\\&\leq c(1+t)^{-\frac{l}{\alpha}}\|\phi\|_{L^2}^2.
\end{split}
\end{eqnarray}
\end{proof}
\begin{lemma}\label{l3.2}
If $\alpha \geq 1$, then
\begin{eqnarray}
\|\Lambda^l G_H(x,t)*\phi\|_{L^q}\leq e^{-Ct}\|\La ^{l+n(\frac{1}{r}-\frac{1}{q})+\delta} \phi\|_{L^r}, \text{ for }\, 1\leq r<2,q>2,
\end{eqnarray} 
\begin{eqnarray}
\|\Lambda^l G_H(x,t)*\phi\|_{L^2}\leq e^{-Ct}\|\La ^{l} \phi\|_{L^2}, 
\end{eqnarray}
where $\delta>0$ is any arbitrary constant, $l>0$.

If $0<\alpha<1$, then 
\begin{eqnarray}\label{3.8jj}
\|\Lambda^l G_H(x,t)*\phi\|_{L^q}\leq (1+t)^{-\frac{\beta}{2\bar{\alpha}}}\|\La ^{\beta+l+n(\frac{1}{r}-\frac{1}{q})+\delta} \phi\|_{L^r}, \text{ for } 1\leq r<2,q>2,
\end{eqnarray} 
\begin{eqnarray}\label{3.9j}
\|\Lambda^l G_H(x,t)*\phi\|_{L^2}\leq (1+t)^{-\frac{\beta}{2\bar{\alpha}}}\|\La ^{\beta+l} \phi\|_{L^2}, 
\end{eqnarray} 
where $\delta>0$ is any arbitrary constant, $\beta>0$, $l>0$, $\bar \alpha=1-\alpha$.
\end{lemma}

\begin{proof} In the high frequency region for $|\xi|\geq 2r$, we have
\begin{eqnarray}
\frac{|\xi|^{2\alpha}}{1+m|\xi|^2}\geq c|\xi|^{2\alpha-2},
\end{eqnarray}
where $c$ is a positive constant depending on $m$.

1) For $\alpha\geq 1$,
\begin{eqnarray}
e^{-\frac{|\xi|^{2\alpha}}{1+m|\xi|^2}t}\leq e^{-ct},
\end{eqnarray} 
where $c$ is a positive constant depending on $R,m$.
Then
\begin{eqnarray}\label{3.12j}
\begin{split}
\|\Lambda^l G_H*\phi\|_{L^q}&\leq c e^{-ct}\Bigl(\int _{|\xi|\geq 2R}\Bigl||\xi|^{l}\hat{\phi}(\xi)\Bigl|^{q'}d\xi\Bigl )^{\frac{1}{q'}}\\
&\leq ce^{-ct}\||\xi|^{-(n+\bar\delta)\frac{r'-q'}{r'q'}}\|_{L^{\frac{r'q'}{r'-q'}}(|\xi|\geq 2R)}\||\xi|^{l+(n+\bar\delta)\frac{r'-q'}{r'q'}}\hat{\phi}\|_{L^{r'}}
\\&\leq ce^{-ct}\|\La^{l+n(\frac{1}{r}-\frac{1}{q})+\delta}\phi\|_{L^r}
\end{split}
\end{eqnarray}
for $1\leq r\leq 2$, $ \frac{1}{r}+\frac{1}{r'}=1$,  $ \frac{1}{q}+\frac{1}{q'}=1$, ${\delta}=(\frac{1}{r}-\frac{1}{q})\bar{\delta}$, $\bar\delta>0$ is an arbitrary constant.
In addition, if $q=2,r=2$, 
\begin{eqnarray}
\|\La^{l}G_H*\phi\|_{L^2}\leq e^{-ct}\||\xi|^l\hat{\phi}\|_{L^2}\leq e^{-ct}\|\La^l \phi\|_{L^2}.
\end{eqnarray}
2) For $0<\alpha <1$, from (\ref{3.9jjjj}), we have
\begin{eqnarray}
\begin{split}
\|\La^l G_H*\phi\|_{L^2} &\leq \||\xi|^l e^{-\frac{|\xi|^{2\alpha}}{1+|\xi|^2}t}\hat{\phi}\|_{L^{2}(|\xi|\geq 2R)}
\\&\leq  \||\xi|^l e^{-c|\xi|^{-2(1-\alpha)}t}\hat{\phi}\|_{L^{2}(|\xi|\geq 2R)}
\\&\leq 
\sup_{|\xi|\geq 2R}|\xi|^{-\beta}e^{-c|\xi|^{-2\bar{\alpha}}t}\bigl(\int_{|\xi|\geq 2R}(|\xi|^{\beta+l}\hat{\phi})^{2} d\xi\Bigl)^{\frac{1}{2}}\\
&\leq (1+t)^{-\frac{\beta}{2\bar{\alpha}}}\||\xi|^{\beta+l}\hat{\phi}\|_{L^{2}(|\xi|\geq 2R)},
\end{split}
\end{eqnarray}
 then (\ref{3.9j}) is established.
Similar to (\ref{3.12j}), we have (\ref{3.8jj}).
\end{proof}

\section{The global existence and  time-decay of the solutions}

Firstly, we consider the estimate for the short wave part $u_H(t)$ and the long wave part $u_L(t)$.  Based on the Fourier transform and (\ref{j3.3}), as in (\ref{3.4jj}), we can define the long wave and short wave decomposition $\bigl(g_L(x,t),g_H(x,t)\bigl)$ for a function $g(x,t)$.\[g_L(x,t)=\chi(D)g(x,t),\,g_H(x)=(1-\chi(D))g(x,t),\]
where $\chi(x) $ is the operator with the symbol $\chi(\xi)$. The long wave part $g_L(x,t)$ and the short wave part $g_H(x,t)$ satisfy 
\begin{eqnarray}\label{4.1jj}
\|\La^lg_L\|_{L^2}\leq C(l)\|g_L\|_{L^2},\,\|g_H\|_{L^2}\leq C(l)\|\La^lg_H\|_{L^2}.
\end{eqnarray}
Obviously, the short wave part $g_H(x,t)$ satisfies a Poincare-like inequality.
Now the solution of (\ref{1.1}) can be divided into two parts by long wave and short wave decomposition:
\[u(x,t)=u_H(x,t)+u_L(x,t).\] 
Thus we can deal with $u_L$ and $ u_H$ respectively. (4.1) help us to have the estimate of lower order term $\|u\|_{L^2}$. Thus we can use the standard energy method to obtain the estimates of $u$.
In the following we give some estimates for the cases $\alpha\geq1$ and $0<\alpha<1$ respectively.
\subsection{a priori estimates for $\alpha\geq 1$}
Note that the Duhamel principle implies that the solution of (\ref{1.1}) satisfies the following integral equation
\begin{eqnarray}\label{3.2jj}
u(t)=G(x,t)*u_0+\int^t_0G(x,t-\tau)B^{-1}\bigl(u^{\theta+1}(\tau)\bigl)d\tau,
\end{eqnarray}
where $B=I-m\Delta$.

Now we estimate $u_L(x,t)$.
For $\alpha>0,$ from (\ref{3.2jj}), we have
\begin{eqnarray}\label{4.2j}
\begin{split}
u_L(x,t)&=G_L(x,t)*u_0+\int^t_0G_L(t-\tau)*B^{-1}\bigl(u^{\theta+1}\bigl)d\tau\\
&=G_L(x,t)*u_0+\int^\frac{t}{2}_0G_L(t-\tau)*B^{-1}\bigl(u^{\theta+1}\bigl)d\tau\\ &\,\,\,\,\,\,\,\,+\int^t_\frac{t}{2}G_L(t-\tau)*B^{-1}\bigl(u^{\theta+1}\bigl)d\tau
\\&=I_0+I_1+I_2.\end{split}
\end{eqnarray}

Define 
\begin{eqnarray}
M_1(t)^2=\sup_{0\leq l\leq s} \sup_{0\leq y\leq t}(1+y)^{\frac{n}{2\alpha}+\frac{l}{\alpha}}\|\Lambda^{l} u(y)\|^2_{L^2}.
\end{eqnarray}
\begin{proposition} 
	If $\alpha\geq 1, s>n/2,\, \theta>\frac{4\alpha}{n},\theta\in \z$, and $u_0\in L^1(\R^n)$,  it follows that 
	\begin{eqnarray}\label{4.28j}
	\|\La^l u_L\|\leq C (1+t)^{-\frac{n}{4\alpha}-\frac{l}{2\alpha}}\bigl( \|u_0\|_{L^1}+M_1(t)^{\theta+1} \bigl)\,\,\text{ for } 0\leq l\leq s. \end{eqnarray}
\end{proposition}
\begin{proof}
According to the assumptions $s>\frac{n}{2}$ and $0\leq l\leq s$, choosing $l >\frac{n}{2}$, it yields
\begin{eqnarray}
	\begin{split}
		&\|u\|_{L^\infty}\leq C \|u\|_{L^2}^{1-a}\|\La ^l u\|_{L^2}^a
		\leq C(1+t)^{-\frac{n}{4\alpha}}M_1(t),
	\end{split}
\end{eqnarray}
where $0<a<1$.
And from Lemma \ref{l2.9}, Lemma \ref{l3.1}, and $\frac{n\theta}{4\alpha}>1$, it gives
\begin{eqnarray}\label{4.4jj}
\begin{split}
\|\La^lI_0\|_{L^2}=\|\La ^lG_L(x,t)*u_0\|_{L^2}\leq C(1+t)^{-\frac{n}{4\alpha}-\frac{l}{2\alpha}}\|u_0\|_{L^1},
\end{split}
\end{eqnarray}
\begin{eqnarray}\label{4.5jj}
\begin{split}
\|\La^lI_1\|_{L^2}&\leq C \int^\frac{t}{2}_0 \|\La^l G_L(x,t-\tau)*B^{-1}(u^{\theta+1})\|_{L^2}d\tau
\\
&\leq C \int^\frac{t}{2}_0 (1+t-\tau)^{-\frac{n}{4\alpha}-\frac{l}{2\alpha}}\|u\|^2_{L^2}\|u\|^{\theta-1}_{L^\infty}d\tau\\&
\leq C (1+\frac{t}{2})^{-\frac{n}{4\alpha}-\frac{l}{2\alpha}}\int^{\frac t2}_0 (1+\tau)^{-\frac{(\theta+1)n}{4\alpha}}M_1(t/2)^{\theta+1}d\tau
\\&\leq C(1+t)^{-\frac{n}{4\alpha}-\frac{l}{2\alpha}}M_1(t)^{\theta+1}, \end{split}
\end{eqnarray}
\begin{eqnarray}\label{4.6jj}
\begin{split}
\|\La^lI_2\|_{L^2}&\leq C \int_\frac{t}{2}^t \| G_L(x,t-\tau)*(\La^l\bigl(u^{\theta+1}\bigl))\|_{L^2}d\tau\\
&\leq C \int^t_\frac{t}{2} \|u\|_{L^\infty}^\theta \|\La ^l u\|_{L^2}d\tau
\\&\leq C \int^t_\frac{t}{2}(1+\tau)^{-\frac{n}{4\alpha}(\theta+1)-\frac{l}{2\alpha}}M_1(t)^{\theta+1} d\tau
\\&\leq C (1+t)^{-\frac{n}{4\alpha}-\frac{l}{2\alpha}}\int^t_\frac{t}{2}(1+\tau)^{-\frac{n\theta}{4\alpha}}M_1(t)^{\theta+1} d\tau
\\&\leq C(1+t)^{-\frac{n}{4\alpha}-\frac{l}{2\alpha}}M_1(t)^{\theta+1}. 
 \end{split}\end{eqnarray}
From (\ref{4.4jj}), (\ref{4.5jj}) and (\ref{4.6jj}), we arrive at (\ref{4.28j}).
\end{proof}
Now we give some estimates about $u_H(x,t)$.
\begin{proposition} If $\alpha\geq 1, s>n/2,\, \theta>\frac{4\alpha}{n},\theta\in \z$, assume $E_0=\|u_0\|_{H^s}+\|u_0\|_{L^1}$ for $u_0\in H^s(\R^n)\bigcap L^1(\R^n)$ sufficiently small, it follows that 
	\begin{eqnarray}\label{4.14jj}
	&\|\La ^l u_H\|^2_{L^2} \leq c(1+t)^{-\frac{n}{2\alpha}-\frac{l}{\alpha}}\Bigl( \|u_0\|^2_{H^l}+\|u_0\|_{L^1}^2+M_1(t)^{2(\theta+1)}\Bigl), \\ \label{4.21-1}
	&	M_1(t)\leq c(\|u_0\|_{H^s}+\|u_0\|_{L^1}).
	\end{eqnarray}
\end{proposition}
\begin{proof}
Applying $1-\chi(D)$ to  both sides of  (\ref{1.1}), it gives
\begin{eqnarray}\label{4.8jj}
\begin{cases}
\frac{\partial u_H}{\partial t}-m\frac{\partial\Delta u_H}{\partial t}+(-\Delta )^{\alpha}u_H= (1-\chi(D))u^{\theta+1},\,\,&x\in\R^n ,t>0\\
u_H(0,x)=(1-\chi(D))u_{0}(x),\,\,&x\in\R^n.
\end{cases}
\end{eqnarray}

Assume \begin{eqnarray}\label{4.15j}
M_1(t)<\va,
\end{eqnarray}then $\|u\|_{L^\infty}<c\va$.

Multiplying (\ref{4.8jj}) by $u_H$,  and integrating over $\R^n$, it yields
\begin{eqnarray}\label{4.9jj}
\frac{1}{2}\frac{d}{dt}\bigl(\|u_H\|^2_{L^2}+m\|\La u_H\|^2_{L^2})+\|\La^\alpha u_H\|^2_{L^2}\leq \int_{\R^n}(1-\chi(D))u^{\theta+1}u_Hdx.
\end{eqnarray}
Since \begin{eqnarray}\label{4.13j}
\begin{split}
\int_{\R^n}(1-\chi(D))u^{\theta+1}u_Hdx&\leq \|(1-\chi(D))u^{\theta+1}\|_{L^2}
\|u_H\|_{L^2}\\
&\leq c \|(1-\chi(D))\La u^{\theta+1}\|_{L^2}
\|\La u_H\|_{L^2}
\\&\leq c\|\La u^{\theta+1}\|_{L^2}
\|\La u_H\|_{L^2}
\\&\leq c\|u\|_{L^\infty}^{\theta}(\|\La u_H\|_{L^2}+\|\La u_L\|_{L^2})\|\La u_H\|_{L^2}
\\&\leq c\va^{\theta}(\|\La u_H\|_{L^2}^2+\|\La u_L\|_{L^2}^2),
\end{split}
\end{eqnarray}
choosing $\va<<1$, from (\ref{4.1jj}), (\ref{4.9jj}) and (\ref{4.13j}) there exists a constant $C>0$ such that 
\[
\frac{d}{dt}\bigl(\|u_H\|^2_{L^2}+m\|\La u_H\|_{L^2}^2)+C(\| u_H\|^2_{L^2}+m\|\La u_H\|_{L^2}^2)\leq c\|\La u_L\|^2_{L^2}
.\]
Then by Gronwall's inequality,  (\ref{4.28j}) and (\ref{4.15j}), we have for $0\leq l\leq 1$,
\begin{eqnarray}
\begin{split}
&\|\La^lu_H\|^2_{L^2}\\ \leq & c\|\La u_H\|_{L^2}^2\\
\leq & ce^{-Ct}\|u_0\|^2_{H^1}+c\int^t_0 e^{-C(t-\tau)}\|\La u_L\|_{L^2}^2d\tau
\\\leq & ce^{-Ct}\|u_0\|^2_{H^1}+c\int^t_0 e^{-C(t-\tau)}(1+\tau)^{-\frac{n}{2\alpha}-\frac{1}{\alpha}}\Bigl(\|u_0\|_{L^1}^2+M_1(t)^{2(\theta+1)}\Bigl)d\tau
\\\leq & c(1+t)^{-\frac{n}{2\alpha}-\frac{1}{\alpha}}\Bigl( \|u_0\|^2_{H^1}+\|u_0\|_{L^1}^2+M_1(t)^{2(\theta+1)}\Bigl) .
\end{split}
\end{eqnarray}
Applying $\La^{l-1}$ to (\ref{4.8jj}),\, and multiplying by $\La^{l-1}u_H$, integrating over $\R^n$ for $l\geq 1$, from Lemma 2.3, we have
\begin{eqnarray}\label{4.12jj}
\begin{split}
&\frac{d}{dt}\bigl(\|\La^{l-1}u_H\|^2_{L^2}+m\|\La^{l} u_H\|^2_{L^2})+\|\La^{l-\bar{\alpha}} u_H\|^2_{L^2}\\\leq & \int_{\R^n}\La^{l-1}( (1-\chi(D))u^{\theta+1})\La^{l-1}u_Hdx\\\leq & \|\La^{l-1}( (1-\chi(D))u^{\theta+1})\|_{L^2}\|\La^{l-1}u_H\|_{L^2}
\\
\leq & \|\La^{l}( u^{\theta+1})\|_{L^2}\|\La^{l}u_H\|_{L^2}\\
\leq & C\|u\|_{L^\infty}^\theta\|\Lambda^{l}u_H\|_{L^2}\|\Lambda^{l}u\|_{L^2}
\\\leq & C\va^\theta (\|\La^{l}u_H\|^2_{L^2}+\|\Lambda^{l}u_L\|^2_{L^2}).
\end{split}
\end{eqnarray}
Since $\bar{\alpha}<0$ by the Poincare-like inequality, we have \begin{eqnarray}\label{4.13jj}
c\|\La^{l-1}u_H\|_{L^2}\leq \|u_H\|_{H^l}\leq C\|\La^{l}u_H\|_{L^2}\leq \|\La^{l-\bar{\alpha}} u_H\|_{L^2}.
\end{eqnarray}
From (\ref{4.12jj}) and (\ref{4.13jj}), we have
\begin{eqnarray}
\begin{split}
&\frac{d}{dt}\bigl(\|\La^{l-1}u_H\|^2_{L^2}+m\|\La^{l} u_H\|^2_{L^2})+C(\| \Lambda^{l-1}u_H\|^2_{L^2}+m\|\La^{l} u_H\|^2_{L^2}))\\ \leq & C\va^\theta (\|\La^{l}u_L\|^2_{L^2}+\|\La^{l}u_H\|^2_{L^2}).
\end{split}
\end{eqnarray}
For $\va<<1$, thus there exists a constant $C>0$ such that 
\begin{eqnarray}
\begin{split}
\frac{d}{dt}\bigl(\|\La^{l-1}u_H\|^2_{L^2}+m\| \La^l u_H\|^2_{L^2})+C(\| \Lambda^{l-1}u_H\|^2_{L^2}+m\| \Lambda^lu_H\|^2_{L^2})\leq c\|\La^l u_L\|^2_{L^2}.
\end{split}
\end{eqnarray}
By Gronwall's 
inequality, (\ref{4.28j}) and (\ref{4.13jj}), we have 
\begin{eqnarray}
\begin{split}
&\|\La ^l u_H\|^2_{L^2}\\ \leq & e^{-Ct}\|u_0\|^2_{H^l}+c\int^t_0 e^{-C(t-\tau)}(1+\tau)^{-\frac{n}{2\alpha}-
	\frac{l}{\alpha}}\Bigl(\|u_0\|_{L^1}^2+M_1(t)^{2(\theta+1)}\Bigl)d\tau
\\ \leq& c(1+t)^{-\frac{n}{2\alpha}-\frac{l}{\alpha}}\Bigl( \|u_0\|^2_{H^l}+\|u_0\|_{L^1}^2+M_1(t)^{2(\theta+1)}\Bigl). 
\end{split}
\end{eqnarray}

Now we have derived  a priori estimates (\ref{4.14jj}) for the case $\alpha\geq 1$.
Thus
\begin{eqnarray}
M_1(t)\leq C \bigl( \|u_0\|_{L^1}+\|u_0\|_{H^s}+M_1(t)^{\theta+1} \bigl)
\end{eqnarray}
under the assumption (\ref{4.15j}). Therefore, by the continuity argument, for $\|u_0\|_{H^s}+\|u_0\|_{L^1}$ sufficiently small,  it implies (\ref{4.21-1}) for $0\leq l\leq s,s>\frac{n}{2}$. 

\end{proof}
\subsection{A priori estimates for  $0<\alpha<1$}

 For the case that $0<\alpha<1$. Define 
 \begin{eqnarray}\label{5.1}
 \begin{split}
 E(t)^2= \sum^{[\frac{s}{2\bar{\alpha}}]-1}_{j=0}\sup_{0\leq y\leq t}(1+y)^{j\bar{\alpha}-\frac{1}{2}\bar{\alpha}}\|\Lambda^{j\bar{\alpha}}u_H(y)\|^2_{H^{s-2j\bar{\alpha}}}, \\ L(t)^2=\sum^{[\frac{s}{2\bar{\alpha}}]-1}_{j=0}\int^t_0 (1+\tau)^{j\bar{\alpha}-\frac{1}{2}\bar{\alpha}}\|\Lambda ^{j\bar{\alpha}} u_H(\tau)\|^2_{H^{s-2j\bar\alpha}}d\tau,
 \end{split}
 \end{eqnarray}
 
 \begin{eqnarray}\label{4.23jjj}
 M_2(t)^2=\sup_{0\leq \beta\leq l} \sup_{0\leq y\leq t}(1+y)^{\frac{n}{2\alpha}+\frac{\beta}{\alpha}}\|\Lambda^{\beta} u(y)\|^2_{L^2}.
 \end{eqnarray}
  \begin{proposition}
  	If $0<\alpha< 1,  \theta>\frac{4\alpha}{n},\theta\in \z$,
  	$([\frac{s}{2\bar \alpha}]-1)\bar\alpha >{\frac{n}{2{\alpha}}+\frac 12\bar\alpha}$
  	and $u_0\in L^1(\R^n)$,  it follows that 
  	\begin{eqnarray}\label{4.28}
  	\|\La^l u_L\|_{L^2} \leq C (1+t)^{-\frac{n}{4\alpha}-\frac{l}{2\alpha}}\bigl( \|u_0\|_{L^1}+(M_2(t)+E(t))^{\theta+1} \bigl). \end{eqnarray} 
  \end{proposition}\begin{proof}
 
Since $([\frac{s}{2\bar \alpha}]-1)\bar\alpha >\frac{n}{2{\alpha}}+\frac 12\bar\alpha>\frac n 2$, from Lemma \ref{l2.9} and the definition of $E(t)$ and $M_2(t)$, we have
 \begin{eqnarray}\label{4.21}
 \begin{split}
 \|u\|_{L^\infty}&\leq C \|u\|_{L^2}^{1-a}\|\La ^{j\bar{\alpha}}  u\|_{L^2}^a \\ &\leq  C  \|u\|_{L^2}^{1-a}\bigl(\|\La ^{j\bar{\alpha}} u_H\|_{L^2}^a+
 \|\La ^{j\bar{\alpha}} u_L\|_{L^2}^a\bigl)\\&
 \leq  C  \|u\|_{L^2}^{1-a}\bigl(\|\La ^{j\bar{\alpha}} u_H\|_{L^2}^a+
 \|u_L\|_{L^2}^a\bigl)\\&
 \leq C (1+t)^{-\frac{n}{4\alpha}(1-a)}M_2(t)^{1-a}[(1+t)^{(-\frac{j\bar\alpha}{2}+\frac{1}{4}\bar\alpha)a}E(t)^a+C(1+t)^{-\frac{n}{4\alpha}a}M_2(t)^a]\\&
 \leq C(1+t)^{-\frac{n}{4\alpha}}(M_2(t)+E(t)),
 \end{split}
 \end{eqnarray}
 where $0<a<1$ and $([\frac{s}{2\bar \alpha}]-1)\bar\alpha\geq j\bar{\alpha} >\max\{\frac{n}{2{\alpha}}+\frac 12\bar\alpha, \frac n 2\}$. 
Similar to $\alpha\geq 1$, we need to estimate $I_0,I_1,I_2$ in (\ref{4.2j}) for $0<\alpha<1$. From Lemma \ref{l3.1} and (\ref{4.21}), it gives
 \begin{eqnarray}\label{4.22}
 \begin{split}
 \|\La^lI_0\|_{L^2}\leq C\|\La ^lG_L(x,t)\|_{L^2}\|u_0\|_{L^1}\leq C(1+t)^{-\frac{n}{4\alpha}-\frac{l}{2\alpha}}\|u_0\|_{L^1},
 \end{split}
 \end{eqnarray}
 \begin{eqnarray}\label{4.23}
 \begin{split}
 \|\La^lI_1\|_{L^2}&\leq C \int^\frac{t}{2}_0 \|\La^l G_L(x,t-\tau)*(u^{\theta+1})\|_{L^2}d\tau\\
 &\leq C \int^\frac{t}{2}_0 (1+t-\tau)^{-\frac{n}{4\alpha}-\frac{l}{2\alpha}}\|u\|^2_{L^2}\|u\|^{\theta-1}_{L^\infty}d\tau
 \\&\leq C(1+t)^{-\frac{n}{4\alpha}-\frac{l}{2\alpha}}M_2(t)^2(M_2(t)+E(t))^{\theta-1}, \end{split}
 \end{eqnarray}
 \begin{eqnarray}\label{4.24}
 \begin{split}
 \|\La^lI_2\|_{L^2}&\leq C \int_\frac{t}{2}^t \| G_L(x,t-\tau)*(\La^l\bigl(u^{\theta+1}\bigl))\|_{L^2}d\tau\\
 &\leq C\int^t_\frac{t}{2}\|\La^l\bigl( u^{\theta+1}\bigl)\|_{L^2}d\tau\\
 &\leq C \int^t_\frac{t}{2} \|\La ^{l}u\|_{L^2}\|u\|^\theta_{L^\infty}d\tau
 \\&\leq C \int^t_\frac{t}{2} (1+\tau)^{-\frac{n}{4\alpha}(\theta+1)-\frac{l}{2\alpha}}M_2(t)(M_2(t)+E(t))^\theta d\tau
 \\&\leq C(1+t)^{-\frac{n}{4\alpha}-\frac{l}{2\alpha}}M_2(t)(M_2(t)+E(t))^\theta. 
 \end{split}\end{eqnarray}
 From (\ref{4.22}), (\ref{4.23}) and (\ref{4.24}), it follows (\ref{4.28}).
\end{proof}
 Now we give some estimates about $u_H(x,t)$ for $0< \alpha<1$.
 \begin{proposition}If $0<\alpha<1$, $[\frac{s}{2\bar{\alpha}}]\geq\frac{n}{2\alpha}+\frac{3}{2}\bar{\alpha}$, $\theta>\frac{4\alpha}{n},\theta\in \z$.
 	Then for 
 	$0\leq l\leq N_0 $ and $E_0$ is sufficiently small, it follows that
 	\begin{eqnarray}\label{4.34}
 	\|\La^lu_H\|_{L^2}\leq C (1+t)^{-\frac{n}{4\alpha}-\frac{l}{2\alpha}}\bigl( \|u_0\|_{H^s}+(M_2(t)+E(t))^{(\theta+1)}\bigl),
 	\end{eqnarray} 
 	\begin{eqnarray}\label{4.38jjj}
 	\|\La^lu\|_{L^2}\leq C (1+t)^{-\frac{n}{4\alpha}-\frac{l}{2\alpha}}\bigl( \|u_0\|_{H^s}+\|u_0\|_{L^1})(M_2(t)+E(t))^{(\theta+1)}\bigl),
 	\end{eqnarray}
 	with $\bar \alpha=1-\alpha$, $ N_0=\alpha \min\{s-\frac{n}{2\alpha}\bar{\alpha}, ([\frac{s}{2\bar{\alpha}}]-1)\bar{\alpha}-\frac{n}{2\alpha}\bar{\alpha}+2\}$ and $E_0=\|u_0\|_{H^s}+\|u_0\|_{L^1}$.
 \end{proposition}
 \begin{proof}
 \begin{eqnarray}
\begin{split}
u_H(x,t)&=G_H(x,t)*u_0+\int^t_0G_H(t-\tau)*B^{-1}\bigl(u^{\theta+1}\bigl)d\tau\\
&=G_H(x,t)*u_0+\int^\frac{t}{2}_0G_H(t-\tau)*B^{-1}\bigl(u^{\theta+1}\bigl)d\tau\\&\,\,\,\,\,\,+\int^t_\frac{t}{2}G_H(t-\tau)*B^{-1}\bigl(u^{\theta+1}\bigl)d\tau
\\&=J_0+J_1+J_2.\end{split}
\end{eqnarray}
In (\ref{3.9j}),  let $\beta_1=\bar{\alpha}(\frac{n}{2\alpha}+\frac{l}{\alpha})$ and $s\geq l+\beta_1$, then it holds
\begin{eqnarray}\label{4.17jj}
\|\La^l J_0\|_{L^2}\leq  (1+t)^{-\frac{\beta_1}{2\bar{\alpha}}}\|\La ^{\beta_1+l} u_0\|_{L^2}\leq (1+t)^{-\frac{n}{4\alpha}-\frac{l}{2\alpha}}\| u_0\|_{H^s}.
\end{eqnarray}
 For $\beta_1+l-2\leq 0$, from Lemma 3.2 and (\ref{4.21}), similarly we have
 \begin{eqnarray}\label{4.18jj}
 \begin{split}
 \|\La^lJ_1\|_{L^2}&\leq C \int^\frac{t}{2}_0 \| G_H(x,t-\tau)*\La^l\bigl(B^{-1}(u^{\theta+1})\bigl)\|_{L^2}d\tau\\&\leq C\int^\frac{t}{2}_0(1+t-\tau)^{-\frac{\beta_1}{2\bar{\alpha}}}\||\xi|^{\beta_1+l-2}\bigl(\mathcal{F}(u^{\theta+1})\bigl)\|_{L^2(|\xi|\geq 2R)}d\tau\\
 &\leq C \int^\frac{t}{2}_0 (1+t-\tau)^{-\frac{n}{4\alpha}-\frac{l}{2\alpha}} \|u\|_{L^\infty}^\theta \| u\|_{L^2}d\tau
 \\&\leq C(1+t/2)^{-\frac{n}{4\alpha}-\frac{l}{2\alpha}} \int^\frac{t}{2}_0 (1+\tau)^{\frac{(\theta+1)n}{4\alpha}} M_2(t)(M_2(t)+E)^{\theta}d\tau
 \\
 &\leq C(1+t)^{-\frac{n}{4\alpha}-\frac{l}{2\alpha}}M_2(t)(M_2(t)+E(t))^{\theta}. \end{split}
 \end{eqnarray}
 For  $0<\beta_1+l-2\leq ([\frac{s}{2\bar{\alpha}}]-1)\bar{\alpha}$,
\begin{eqnarray}\label{4.32}
\begin{split}
\|\La^lJ_1\|_{L^2}&\leq C \int^\frac{t}{2}_0 \| G_H(x,t-\tau)*\La^l\bigl(B^{-1}(u^{\theta+1})\bigl)\|_{L^2}d\tau\\
&\leq C\int^\frac{t}{2}_0(1+t-\tau)^{-\frac{\beta_1}{2\bar{\alpha}}}\|\La^{\beta_1+l-2}\bigl(u^{\theta+1}\bigl)\|_{L^2}d\tau\\
&\leq  C(1+\frac{t}{2})^{-\frac{n}{4\alpha}-\frac{l}{2\alpha}}\int^\frac{t}{2}_0  \|u\|_{L^\infty}^\theta (\|\La^{l+\beta_1-2} u_L\|_{L^2}+\|\La^{l+\beta_1-2} u_H\|_{L^2})d\tau\\
&\leq  C(1+t)^{-\frac{n}{4\alpha}-\frac{l}{2\alpha}}\int^\frac{t}{2}_0 \|u\|_{L^\infty}^\theta (\| u_L\|_{L^2}+\|\La^{l+\beta_1-2} u_H\|_{L^2})d\tau
\\&\leq C(1+t)^{-\frac{n}{4\alpha}-\frac{l}{2\alpha}}(M_2(t)+E(t))^{\theta+1}, \end{split}
\end{eqnarray}
\begin{eqnarray}\label{4.19jjj}
\begin{split}
\|\La^lJ_2\|_{L^2}&\leq C \int_\frac{t}{2}^t \| G_H(x,t-\tau)*(\La^l\bigl(B^{-1}u^{\theta+1}\bigl))\|_{L^2}d\tau\\
&\leq C\int^t_\frac{t}{2}\|\La^l\bigl(B^{-1}
 u^{\theta+1}\bigl)\|_{L^2}d\tau\\
 &\leq C\int^t_\frac{t}{2}\|\La^l\bigl(
 u^{\theta+1}\bigl)\|_{L^2}d\tau\\
&\leq C \int^t_\frac{t}{2} \|u\|_{L^\infty}^\theta \|\La ^l u\|_{L^2}d\tau
\\&\leq C \int^t_\frac{t}{2} (1+\tau)^{-\frac{n}{4\alpha}(\theta+1)-\frac{l}{2\alpha}}M_2(t)(M_2(t)+E(t))^\theta d\tau
\\&\leq C(1+t)^{-\frac{n}{4\alpha}-\frac{l}{2\alpha}}M_2(t)(M_2(t)+E(t))^\theta .
\end{split}\end{eqnarray}
From (\ref{4.17jj})-(\ref{4.19jjj}), then it gives 
\begin{eqnarray}\label{4.34}
\|\La^lu_H\|_{L^2}\leq C (1+t)^{-\frac{n}{4\alpha}-\frac{l}{2\alpha}}\bigl( \|u_0\|_{H^s}+(M_2(t)+E(t))^{(\theta+1)}\bigl).
\end{eqnarray}
Thus by (\ref{4.28}) and (\ref{4.34}),  then we have (\ref{4.38jjj})

\end{proof}
\begin{proposition}
	For $0<\alpha<1$, we claim that  for $0\leq k\leq [\frac{s}{2\bar\alpha}]-1$,
	\begin{eqnarray}\label{a}
	\begin{split}
	&(1+t)^{(k-\frac{1}{2})\bar{\alpha}}\|\Lambda^{k\bar{\alpha}}u_H\|_{H^{s-2k\bar\alpha}}^2+c\int^t_0 (1+\tau)^{(k-\frac{1}{2})\bar{\alpha}}\|\Lambda^{(k+1)\bar{\alpha}}u_H\|_{H^{s-2(k+1)\bar\alpha}}^2d\tau\\   \leq & c\|u_0\|_{H^s}^2+c(M_2(t)^2+E(t)^2)^{\theta}(M_2(t)^2+L(t)^2).
	\end{split}
	\end{eqnarray}
\end{proposition}
\begin{proof}
 (\ref{4.8jj}) can be rewritten as the following
\begin{eqnarray}\label{5.3jj}
\begin{cases}
\dis\frac{\partial u_H}{\partial t}+B^{-1}(-\Delta )u_H=B^{-1} (1-\chi(D))u^{\theta+1},\,\,&x\in\R^n ,t>0\\
u_H(0,x)=u_0(x)(1-\chi(D)),\,\,&x\in\R^n.
\end{cases}
\end{eqnarray}

From (\ref{4.1jj}), (\ref{4.9jj}) and (\ref{4.12jj}), let $l=s$, recall that $$s-\bar{\alpha}\geq \bar\alpha,\|\La^{s-\bar\alpha}u_H\|_{L^2}\geq c\|\La^{\bar\alpha} u_H\|_{L^2},$$ it yields
\begin{eqnarray}\label{5.5}
\begin{split}
&\frac{1}{2}\frac{d}{dt}\bigl(\|\La ^{s-1}u_H\|^2_{L^2}+m\|\La ^{s}u_H\|^2_{L^2}+\|u_H\|^2_{L^2}+m\|\La u_H\|^2_{L^2}\bigl)+C\|\La^{\bar{\alpha}}u_H\|_{H^{s-2\bar{\alpha}}}^2
\\ \leq & c\|u\|^{\theta}_{L^\infty}(\|u_L\|_{L^2}^2+\| \La^{s-1}u_H\|_{L^2}^2).
\end{split}
\end{eqnarray} 
Multiply (\ref{5.5}) by  $(1+t)^{-\frac{1}{2}\bar{\alpha}}$, integrate with respect to $t$, then
\begin{eqnarray}
\begin{split}
&(1+t)^{-\frac{1}{2}\bar{\alpha}}(\|u_H\|^2_{H^s}+\|u_H\|_{H^{s-1}}^2)+c\int_{0}^{t}(1+\tau)^{-\frac{\bar{\alpha}}{2}}\|\Lambda ^{\bar{\alpha}}u_H\|_{H^{s-2\bar{\alpha}}}^2d\tau \\
\leq & c\|u_0\|^2_{H^s} -\frac{1}{2}\bar{\alpha}\int^t_0(1+\tau)^{-\frac{\bar{\alpha}}{2}-1}(\|u_H\|_{H^s}^2+\|u_H\|_{H^{s-1}}^2)d\tau \\&+c\int^t_0 (1+\tau)^{-\frac{\bar{\alpha}}{2}}\|u\|^{\theta}_{L^\infty}(\|u_L\|_{L^2}^2+\| \La^{s-1}u_H\|_{L^2}^2)d\tau.
\end{split}
\end{eqnarray}
Since $\bar{\alpha}>0$, we have
\begin{eqnarray}
\begin{split}
&(1+t)^{-\frac{1}{2}\bar{\alpha}}(\|u_H\|_{H^s}^2+\|u_H\|_{H^{s-1}}^2)+c\int_{0}^{t}(1+\tau)^{-\frac{\bar{\alpha}}{2}}\|\Lambda ^{\bar{\alpha}}u_H\|^2_{H^{s-2\bar{\alpha}}}d\tau \\
\leq & c\|u_0\|^2_{H^s}+c\int^t_0 (1+\tau)^{-\frac{\bar{\alpha}}{2}}\|u\|^{\theta}_{L^\infty}(\|u_L\|_{L^2}^2+\| \La^{s-1}u_H\|_{L^2}^2)d\tau \\
\leq & c\|u_0\|^2_{H^s}+c\Bigl(M_2(t)^2+L(t)^2\Bigl)\Bigl(M_2(t)+E(t)\Bigl)^{\theta}.
\end{split}
\end{eqnarray}
Thus for $k=0$, (\ref{a}) holds.

Multipling  (\ref{5.3jj}) by  $\Lambda^{2\beta} u $, and integrating with respect to $x$, by Lemma \ref{l2.9} and (\ref{4.1jj}), we arrive at
\begin{eqnarray}\label{4.41}
\begin{split}
\frac{1}{2}\frac{d}{dt}\|\Lambda ^\beta u_H\|^2 +\|\frac{|\xi|^{\beta+\alpha}}{(1+m|\xi|^2)^{1/2}}\hat u_H\|^2_{L^2}& =  \int_{\R^n} B^{-1} (1-\chi(D))u^{\theta+1}\Lambda^{2\beta} u_H dx
\\&
\leq c
\|\Lambda^{\beta} u^{\theta+1}\|_{L^2}\|\Lambda ^{\beta} u_H\|_{L^2}
\\&\leq c\|u\|_{L^\infty}^{\theta}\Bigl ( \|\La^{\beta}u_L\|_{L^2}^2+\|\Lambda ^{\beta} u_H\|_{L^2}^2\Bigl).
\end{split}
\end{eqnarray}
For $|\xi|\geq 2R$ and $\beta\geq\bar{\alpha}$, there exists a constant $c$ such that 
\[\frac{|\xi|^{2\beta+2\alpha}}{1+m|\xi|^2}\geq c(1+|\xi|^{2\beta+2\alpha-2}), \]
then
\begin{eqnarray}
\|\frac{|\xi|^{2\beta+2\alpha}}{1+m|\xi|^2}\hat u_H\|_{L^2}\geq c \| u_H\|_{H^{\alpha+\beta-1}},
\end{eqnarray}
thus 
\begin{eqnarray}\label{5.7}
\begin{split}
\frac{d}{dt}\|\La ^{\beta}u_H\|^2_{L^2}+c\|u_H\|_{H^{\beta+\alpha-1}}^2
&\leq c \|u\|_{L^\infty}^{\theta}\Bigl ( \|\La^{\beta}u_L\|_{L^2}^2+\|\Lambda ^{\beta} u_H\|_{L^2}^2\Bigl).
\end{split}
\end{eqnarray}

Recall that $s>2(j+1)\bar{\alpha}$, $|\xi|\geq 2r$, it holds
then
\begin{eqnarray}\label{5.9}
\begin{split}
\|u_H\|^2_{H^{s-(j+1)\bar{\alpha}}}\geq \|\La^{s-(j+1)\bar{\alpha}}u_H\|&\geq \frac{1}{2}(\|\La^{(j+1)\bar{\alpha}}u_H\|^2+\|\La^{s-(j+1)\bar{\alpha}}u_H\|^2)\\&\geq C \|\La^{(j+1)\bar{\alpha}}u_H\|_{H^{s-2(j+1)\bar{\alpha}}}^2.
\end{split}
\end{eqnarray}
 Let $\beta=s-j\bar{\alpha}$ in (\ref{5.7}), from (\ref{5.9}), it yields
\begin{eqnarray}\label{5.10}
\begin{split}
&\frac{d}{dt}\|\La ^{s-j\bar{\alpha}}u_H\|^2_{L^2}+c\|\La^{(j+1)\bar{\alpha}}u_H\|_{H^{s-2(j+1)\bar{\alpha}}}^2\\
\leq &c\|u\|^{\theta}_{L^\infty
}\Bigl ( \|\Lambda ^{s-j\bar{\alpha}}u_L\|_{L^2}^2+\|\Lambda ^{s-j\bar{\alpha}} u_H\|_{L^2}^2\Bigl )\\
\leq &c\|u\|^{\theta}_{L^\infty
}\Bigl ( \|u_L\|_{L^2}^2+\|\Lambda ^{s-j\bar{\alpha}} u_H\|_{L^2}^2\Bigl ).
\end{split}
\end{eqnarray}
Multiplying (\ref{5.10}) by  $(1+t)^{j\bar{\alpha}-\frac{1}{2}\bar{\alpha}}$ and integrating with respect to $t$, from $$\|\Lambda^{j\bar{\alpha}}u_H\|_{H^{s-2j\bar\alpha}}^2\geq \|\Lambda ^{s-j\bar{\alpha}}u_H\|_{L^2}^2\geq c\|\Lambda^{j\bar{\alpha}}u_H\|_{H^{s-2j\bar\alpha}}^2,$$ we know that 
\begin{eqnarray}\label{5.11}
\begin{split}
&(1+t)^{j\bar{\alpha}-\frac{1}{2}\bar{\alpha}}\|\La^{j\bar{\alpha}}u_H\|^2_{H^{s-2j\bar{\alpha}}}+c\int^t_0(1+\tau)^{j\bar{\alpha}-\frac{1}{2}\bar{\alpha}}\|\Lambda ^{(j+1)\bar{\alpha}} u_H\|^2_{H^{s-2(j+1)\bar{\alpha}}}d\tau\\
& \leq c \int^t_0 (1+\tau)^{j\bar{\alpha}-\frac{1}{2}\bar{\alpha}}\|u\|^{\theta}_{L^\infty
}\Bigl ( \|u_L\|_{L^2}^2+\|\Lambda ^{s-j\bar{\alpha}} u_H\|_{L^2}^2\Bigl )d\tau \\&
\,\,+(j\bar{\alpha}-\frac{1}{2}\bar{\alpha}-1)\int ^{t}_0 (1+\tau)^{j\bar{\alpha}-\frac{1}{2}\bar{\alpha}-1} \|\Lambda ^{j\bar{\alpha}} u_H\|^2_{H^{s-2j\bar{\alpha}}}d\tau +c\|u_0\|_{H^s}^2,
\end{split}
\end{eqnarray}
and
\begin{eqnarray}\label{5.12}
\begin{split}
& \int^t_0 (1+\tau)^{j\bar{\alpha}-\frac{1}{2}\bar{\alpha}}\|u\|^{\theta}_{L^\infty
}\Bigl ( \|u_L\|_{L^2}^2+\|\Lambda ^{s-j\bar{\alpha}} u_H\|_{L^2}^2\Bigl )d\tau 
\\& \leq c (M_2(t)+E(t))^{\theta}  \int^t_0 (1+\tau)^{-\frac{n\theta}{4\alpha}}(1+\tau)^{j\bar{\alpha}-\frac{\bar{\alpha}}{2}}\Bigl (  \|u_L\|_{L^2}^2+\|\Lambda ^{s-j\bar{\alpha}} u_H\|_{L^2}^2\Bigl )d\tau
\\&\leq
c \Bigl(M_2(t)^2+L(t)^2\Bigl)\Bigl(M_2(t)+E(t)\Bigl)^\theta.
\end{split}
\end{eqnarray}

If $1\leq k=j\leq \frac 12+\frac{1}{\bar\alpha}$, by (\ref{5.11}) and (\ref{5.12}), (\ref{a}) is obviously true.

If $k\geq \frac 12+\frac{1}{\bar\alpha}$,  
assume that for $k=j$, (\ref{a}) holds. Then
\begin{eqnarray}\label{4.6}
\begin{split}
&\int^t_0 (1+t)^{j\bar{\alpha}-\frac{1}{2}\bar{\alpha}}\|\La^{(j+1)\bar{\alpha}}u_H\|_{H^{s-2(j+1)\bar\alpha}}^2d\tau \\ \leq & c \Bigl(M_2(t)^2+L(t)^2\Bigl)\Bigl(M_2(t)+E(t)\Bigl)^\theta+c\|u_0\|_{H^s}^2.
\end{split}
\end{eqnarray}
and from
\begin{eqnarray}
\begin{split}
\int^t_0 &(1+\tau)^{(j+1)\bar{\alpha}-\frac{1}{2}\bar{\alpha}-1}\|\La^{(j+1)\bar{\alpha}}u_H\|_{H^{s-2(j+1)\bar{\alpha}}}^2d\tau \\ \leq & c \int^t_0 (1+t)^{j\bar{\alpha}-\frac{1}{2}\bar{\alpha}}\|\La^{(j+1)\bar{\alpha}}u_H\|_{H^{s-2(j+1)\bar\alpha}}^2d\tau,
\end{split}
\end{eqnarray}
it is easy to see that (\ref{a}) is proved for $k=j+1$. The general case can be shown by using an induction argument.
\end{proof}
\begin{proposition}
	If $0<\alpha<1$, $[\frac{s}{2\bar{\alpha}}]\geq\frac{n}{2\alpha}+\frac{3}{2}\bar{\alpha}$, $\theta>\frac{4\alpha}{n},\theta\in \z$.
	 	Then for 
	 	$0\leq l\leq N_0 $ and $E_0$ is sufficiently small, it follows that
	 \begin{equation}\label{5.15}	M_2(t)\leq cE_0,
	 \end{equation}
	  where $N_0,E_0$ are defined as in Theorem 1.2
\end{proposition}
\begin{proof}
Let $Y(t)^2=E(t)^2+L(t)^2+M_2(t)^2$,  from (\ref{4.38jjj}) and (\ref{a}) we have $$Y(t)^2\leq c (\|u\|_{L^1}^2+\|u_0\|_{H^s}^2)+cY(t)^{2\theta+2}+cY(t)^{\theta+1 }.$$ For $\|u_0\|_{L^1}+\|u_0\|_{H^s}$ sufficiently small, by continuity, we have \begin{eqnarray}
Y(t)\leq CE_0.
\end{eqnarray}
Thus the proof is complete.
\end{proof}
\subsection{ Proof of Theorem 1.1 and Theorem 1.2} 
The local existence of the solution for (1.1) can be proved by using the standard method （(Refer to Lemma  3.3 in \cite{LYW}, or Theorem 2.2 in \cite{K2}), we omit its details. Then combining the local solution with  estimates (\ref{4.21-1}) and (\ref{5.15}), we obtain the global existence of the solution to the Cauchy problem for $\alpha>0$ if $E_0$ is sufficiently small. This proves Theorem 1.1 and Theorem 1.2.


\end{document}